\documentclass[11pt]{article}

\title{\vspace{-40pt}Equivariant group presentations and the second homology group of the Torelli group}
\author{Martin Kassabov\thanks{Supported in part by NSF grant DMS-1601406} \and Andrew Putman\thanks{Supported in part by NSF grants DMS-1737434 and DMS-1811322.}}
\date{}

\usepackage{amsmath,amssymb,amsthm,amscd,amsfonts}
\usepackage{mathtools}
\usepackage{epsfig,pinlabel,paralist}
\usepackage[vmargin=1in, hmargin=1.25in]{geometry}
\usepackage[font=small,format=plain,labelfont=bf,up,textfont=it,up]{caption}
\usepackage{type1cm}
\usepackage{calc}
\usepackage{enumerate}
\usepackage{enumitem}
\usepackage{bm}
\usepackage[all,cmtip]{xy}
\usepackage{eucal}
\usepackage{paralist}
\usepackage{lmodern}
\usepackage[T1]{fontenc}
\usepackage{etoolbox}
\usepackage{leftidx}

\usepackage[bookmarks, bookmarksdepth=2, colorlinks=true, linkcolor=blue, citecolor=blue, urlcolor=blue]{hyperref}

\makeatletter
\DeclareFontFamily{OMX}{MnSymbolE}{}
\DeclareSymbolFont{MnLargeSymbols}{OMX}{MnSymbolE}{m}{n}
\SetSymbolFont{MnLargeSymbols}{bold}{OMX}{MnSymbolE}{b}{n}
\DeclareFontShape{OMX}{MnSymbolE}{m}{n}{
    <-6>  MnSymbolE5
   <6-7>  MnSymbolE6
   <7-8>  MnSymbolE7
   <8-9>  MnSymbolE8
   <9-10> MnSymbolE9
  <10-12> MnSymbolE10
  <12->   MnSymbolE12
}{}
\DeclareFontShape{OMX}{MnSymbolE}{b}{n}{
    <-6>  MnSymbolE-Bold5
   <6-7>  MnSymbolE-Bold6
   <7-8>  MnSymbolE-Bold7
   <8-9>  MnSymbolE-Bold8
   <9-10> MnSymbolE-Bold9
  <10-12> MnSymbolE-Bold10
  <12->   MnSymbolE-Bold12
}{}

\let\llangle\@undefined
\let\rrangle\@undefined
\DeclareMathDelimiter{\llangle}{\mathopen}%
                     {MnLargeSymbols}{'164}{MnLargeSymbols}{'164}
\DeclareMathDelimiter{\rrangle}{\mathclose}%
                     {MnLargeSymbols}{'171}{MnLargeSymbols}{'171}

\apptocmd{\thebibliography}{\raggedright}{}{}

\numberwithin{equation}{section}
 
\theoremstyle{plain}
\newtheorem{theorem}{Theorem}[section]
\newtheorem{maintheorem}{Theorem}

\newtheorem{lemma}[theorem]{Lemma}

\newtheorem{claimsa}{Claim}[section]
\newtheorem{claimsb}{Claim}[section]

\newtheorem*{case}{Case}

\theoremstyle{definition}

\newtheorem{defn}[theorem]{Definition}
\newenvironment{definition}[1][]{\begin{defn}[#1]\pushQED{\qed}}{\popQED \end{defn}}

\theoremstyle{remark}
\newtheorem{rmk}[theorem]{Remark}
\newenvironment{remark}[1][]{\begin{rmk}[#1] \pushQED{\qed}}{\popQED \end{rmk}}

\newtheorem{eg}[theorem]{Example}
\newenvironment{example}[1][]{\begin{eg}[#1] \pushQED{\qed}}{\popQED \end{eg}}

\DeclareMathOperator{\Image}{Im}

\DeclareMathOperator{\Mod}{Mod}

\newcommand\Torelli{\ensuremath{{\mathcal I}}}

\DeclareMathOperator{\Sp}{Sp}

\newcommand\Z{\ensuremath{\mathbb{Z}}}
\newcommand\Q{\ensuremath{\mathbb{Q}}}
\newcommand\N{\ensuremath{\mathbb{N}}}

\DeclareMathOperator{\HH}{H}

\DeclareMathOperator{\Tor}{Tor}

\DeclareMathOperator{\Aut}{Aut}

\newcommand\Present[2]{\ensuremath{\langle \text{#1 $|$ #2} \rangle}}

\newcommand\Set[2]{\ensuremath{\{\text{#1 $|$ #2}\}}}
\newcommand\Norm[1]{\ensuremath{\llangle #1 \rrangle}}

\newcommand\Figure[4]{
\begin{figure}[t]
\centering
\centerline{\psfig{file=#2,scale=#4}}
\caption{#3}
\label{#1}
\end{figure}}

\DeclareMathOperator{\FP}{FP}
\DeclareMathOperator{\tor}{tor}
\DeclareMathOperator{\rank}{rank}
\DeclareMathOperator{\conj}{conj}
\newcommand\Sym{\ensuremath{\mathfrak{S}}}
\newcommand\tM{\ensuremath{\widetilde{M}}}
\newcommand{\p}[1]{{\bf #1.}}

\begin{document}

\vspace{-10pt}
\maketitle

\vspace{-24pt}
\begin{abstract}
We develop a theory of equivariant group presentations and relate them to the
second homology group of a group.  Our main application says that
the second homology group of the Torelli subgroup of the mapping class group
is finitely generated as a $\Z[\Sp_{2g}(\Z)]$-module.
\end{abstract}

\section{Introduction}
\label{section:introduction}

If $G$ is a finitely presentable group, then $\HH_2(G)$ is a finitely generated abelian
group.  Recent work on representation stability has focused
on groups that do not satisfy finiteness conditions like finite presentability, but where
the lack of finiteness is explained by the action of a larger group.  In this spirit,
for a group $G$ acted upon by a group $\Gamma$, we
introduce ``finite
$\Gamma$-equivariant presentations'' for $G$ and show that in many situations, having
such a presentation implies that $\HH_2(G)$ is finitely generated as a $\Z[\Gamma]$-module.
As an application, we prove a conjecture of Church--Farb about the second homology group
of the Torelli group.

\p{Torelli group}
Let $\Sigma_g^b$ be a compact oriented genus $g$ surface with $b$ boundary components and let $\Mod_g^b$ be
its mapping class group, i.e.\ the group of isotopy classes of orientation-preserving diffeomorphisms
$f$ of $\Sigma_g^b$ with $f|_{\partial \Sigma_g^b} = \text{id}$.  For $b \in \{0,1\}$, Poincar\'{e} duality implies
that the algebraic intersection form on $\HH_1(\Sigma_g^b;\Z)$ is a $\Mod_g^b$-invariant symplectic form.
The $\Mod_g^b$-action on $\HH_1(\Sigma_g^b;\Z)$ thus gives a representation
$\Mod_g^b \rightarrow \Sp_{2g}(\Z)$ whose kernel $\Torelli_g^b$ is the {\em Torelli group}.
This is summarized in the short exact sequence
\[1 \longrightarrow \Torelli_g^b \longrightarrow \Mod_g^b \longrightarrow \Sp_{2g}(\Z) \longrightarrow 1.\]
See \cite{JohnsonSurvey, FarbMargalitPrimer, PutmanSurvey} for surveys about the mapping class group and Torelli group.\footnote{Defining the Torelli group on a surface with multiple
boundary components is subtle and there are several possibilities; see 
\cite{PutmanCutPaste}.  We will thus restrict ourselves to surfaces with at most
$1$ boundary component.}

\p{Combinatorial group theory}
The group $\Mod_g^b$ has strong finiteness properties.  For instance,
it is finitely presentable \cite{McCoolFinite} and
all of its homology groups are finitely generated \cite{HarerFinite}.  
Since $\Torelli_g^b$ is an infinite-index subgroup of $\Mod_g^b$, there is no formal reason for it to inherit any of these
finiteness properties.  Indeed, McCullough--Miller \cite{McCulloughMiller} proved that $\Torelli_2^b$ 
is not even finitely generated.  Mess \cite{Mess} strengthened this by showing that
$\Torelli_2$ is an infinite rank free group.  
However, a remarkable theorem of Johnson \cite{JohnsonFinite} says
that $\Torelli_g^b$ is finitely generated for $g \geq 3$.  This has been strengthened in various ways; for
instance, more efficient generating sets can be found in \cite{PutmanSmallGenset} and generation
results for deeper subgroups can be found in \cite{ChurchPutman, ChurchErshovPutman, ErshovHe}.  However, it is
not known whether $\Torelli_g^b$ is finitely presentable for $g \geq 3$.

\p{Homology}
Even the easier question of whether $\HH_2(\Torelli_g^b)$ is finitely generated 
is open (though large pieces of it have been calculated by Hain \cite{HainInfinitesimal}
and by Brendle--Farb \cite{BrendleFarb}), so it is natural to study weaker finiteness properties.
The conjugation action of $\Mod_g^b$ on $\Torelli_g^b$ descends to an action of $\Sp_{2g}(\Z)$ on 
$\HH_k(\Torelli_g^b)$.  Church--Farb \cite{ChurchFarbRepStability} made a series
of conjectures about this action which 
assert that it exhibits various forms of ``representation stability''.  Precisely stating all
of their conjectures would take us too far afield, so we will only do so for the one that we prove.
 
The group $\HH_1(\Torelli_g^b)$ was calculated
by Johnson \cite{JohnsonAbel}, and this calculation shows that all of Church--Farb's conjectures hold for it.
All subsequent work has focused on $\HH_2(\Torelli_g^b)$.  Boldsen--Dollerup \cite{BoldsenDollerup}
proved that $\HH_2(\Torelli_g^1;\Q)$ satisfies a regularity condition introduced by
Church--Farb called ``surjective representation stability''.  Miller--Patzt--Wilson \cite{MillerPatztWilson}
strengthened this by showing that $\HH_2(\Torelli_g^1;\Q)$ is ``centrally stable'' in the sense
of \cite{PutmanSam} (which generalizes to groups like $\Sp_{2g}(\Z)$ ideas from \cite{PutmanCongruence}
concerning the symmetric group).

Church--Farb also conjectured that $\HH_k(\Torelli_g^b)$
is finitely generated as a $\Z[\Sp_{2g}(\Z)]$-module for $g \gg 0$.  In other words,
there exists a finite set $S \subset \HH_k(\Torelli_g^b)$ such
that the $\Sp_{2g}(\Z)$-orbit of $S$ spans $\HH_k(\Torelli_g^b)$.  Our first
main theorem verifies this for $k=2$.
 
\begin{maintheorem}
\label{theorem:mainfg}
$\HH_2(\Torelli_g^b)$ is finitely generated as a $\Z[\Sp_{2g}(\Z)]$-module for $g \geq 3$ and $b \in \{0,1\}$.
\end{maintheorem}

\begin{remark}
Day--Putman \cite{DayPutmanH2IA} proved an analogue of Theorem \ref{theorem:mainfg} for the
Torelli subgroup of $\Aut(F_n)$.  Our proof of Theorem \ref{theorem:mainfg} shares some features
with \cite{DayPutmanH2IA}; for instance, both start with infinite presentations
for the groups in question.  However, the proofs are fundamentally different.  For instance, \cite{DayPutmanH2IA}
identifies an explicit generating set for $\HH_2$, while our proof of Theorem \ref{theorem:mainfg} is
inherently non-constructive.  
\end{remark}

\p{Equivariant presentations}
For a group $G$, the group $\HH_2(G)$ is connected to the relations in a presentation for $G$.
We derive Theorem \ref{theorem:mainfg} from a special kind of presentation for $\Torelli_g^b$
that incorporates the action of $\Mod_g^b$.  For a set $S$, let $F(S)$ be the free group on $S$.

\begin{definition}
Let $G$ and $\Gamma$ be groups such that $\Gamma$ acts on $G$.  
A {\em finite $\Gamma$-equivariant presentation} for $G$ consists of a pair $(S_0,R_0)$ as follows:
\vspace{-\baselineskip}
\begin{compactitem}
\item $S_0 \subset G$ is a finite set whose orbits $S \coloneqq \Gamma \cdot S_0$ generate $G$.
\item $R_0 \subset F(S)$ is a finite set of relations for $G$ whose orbits
$R \coloneqq \Gamma \cdot R_0$ form a complete set of relations for $G$.  Here $\Gamma$ acts on $F(S)$ via its action on $S$.\qedhere
\end{compactitem}
\end{definition}

\begin{example}
Let 
\[G = \bigoplus_{n \in \N} \Z/2\] 
and let $\Gamma$ be the symmetric group on the set $\N$.  The group $\Gamma$ acts on $G$ via its action on $\N$.
For $n \in \N$, let $s_n \in G$ be the generator of the $n^{\text{th}}$ summand.  Then $G$ has a finite
$\Gamma$-equivariant presentation $(S_0,R_0)$ with $S_0 = \{s_1\}$ and $R_0 = \{s_1^2, [s_1,s_2]\}$.
\end{example}

\begin{example}
For some $n \geq 3$, let $\Sym_{n+1}$ be the symmetric group on $(n+1)$ letters $\{1,\ldots,n+1\}$ and
let $\Gamma \subset \Sym_{n+1}$ be any subgroup acting $3$-transitively on $\{1,\ldots,n\}$
and fixing $n+1$.  The group $\Gamma$ acts on $\Sym_{n+1}$ by conjugation.
For $1 \leq i \leq n$, let $s_i \in \Sym_{n+1}$ be the transposition
$(i,n+1)$.  Then $\Sym_{n+1}$ has a finite $\Gamma$-equivariant presentation
$(S_0,R_0)$ with $S_0 = \{s_1\}$ and 
$R_0 = \{s_1^2, (s_1 s_2)^3, (s_1 s_2 s_3)^4\}$.  Indeed, unpacking the definition of
a $\Gamma$-equivariant presentation, we see that this corresponds to the ordinary
group presentation
\[\Present{$s_1,\ldots,s_n$}{$s_i^2$, $(s_i s_j)^3$, $(s_i s_j s_k)^4$}.\]
Here distinct indices represent distinct numbers.  This presentation of $\Sym_{n+1}$ is
a small variant on a presentation of Burnside \cite[p.\ 464]{Burnside} and Miller \cite[p.\ 366]{Miller}; see
\cite[\S 2.2]{GKKL} for more details.
\end{example}

\begin{remark}
Finite $\Gamma$-equivariant presentations are related to but different from the L-presentations defined by
Bartholdi \cite{Bartholdi} and used to study the Torelli subgroup of $\Aut(F_n)$ by
Day--Putman \cite{DayPutmanH2IA}.
\end{remark}

We prove that subject to some technical conditions, a group $G$ with a finite $\Gamma$-equivariant presentation
has $\HH_2(G)$ finitely generated as a $\Z[\Gamma]$-module.  A group $\Gamma$ is of type $\FP_{n}$ if
the trivial $\Z[\Gamma]$-module $\Z$ has a length $n$ partial resolution by finitely generated projective modules.
This holds, for instance, if $\Gamma$ has a $K(\Gamma,1)$ whose $n$-skeleton is compact.

\begin{maintheorem}
\label{theorem:mainabstract}
Let $G$ and $\Gamma$ be groups such that $\Gamma$ acts on $G$.  Assume the following:
\vspace{-\baselineskip}
\begin{compactitem}
\item $G$ has a finite $\Gamma$-equivariant presentation $(S_0,R_0)$.
\item $\HH_1(G)$ is finitely generated as an abelian group.
\item $\Gamma$ is of type $\FP_2$.
\item The $\Gamma$-stabilizers of all elements of $S_0$ are finitely generated.
\end{compactitem}
\vspace{-\baselineskip}
Then $\HH_2(G)$ is finitely generated as a $\Z[\Gamma]$-module.
\end{maintheorem}

\begin{remark}
If $\Gamma$ is of type $\FP_n$ for $n > 2$ and the $\Gamma$-stabilizers of all elements
of $S_0$ are of type $\FP_{n-1}$, then our proof of Theorem \ref{theorem:mainabstract}
almost proves that $\HH_2(G)$ is a $\Z[\Gamma]$-module of type $\FP_n$.
The only thing that goes wrong is Claim \ref{claimsa.4} from the proof, which would
require some kind of higher regularity for the relations that
seems hard to verify in practice.
\end{remark}

\p{Back to Torelli}
The group $\Mod_g^b$ acts on $\Torelli_g^b$ by conjugation, and we prove the following.

\begin{maintheorem}
\label{theorem:torellipres}
$\Torelli_g^b$ has a finite $\Mod_g^b$-equivariant presentation for $g \geq 3$ and $b \in \{0,1\}$.
\end{maintheorem}

\vspace{-\baselineskip}
Given this, we can apply Theorem \ref{theorem:mainabstract} to deduce Theorem \ref{theorem:mainfg}.
Indeed, the other conditions of Theorem \ref{theorem:mainabstract} are satisfied:
\vspace{-\baselineskip}
\begin{compactitem}
\item Johnson \cite{JohnsonFinite} proved that $\Torelli_g^b$ is finitely generated, so $\HH_1(\Torelli_g^b)$ is
a finitely generated abelian group.  Johnson later calculated $\HH_1(\Torelli_g^b)$ in \cite{JohnsonAbel}.
\item Harer \cite{HarerFinite} proved that $\Mod_g^b$ is of type $\FP_{\infty}$, so it certainly is
of type $\FP_2$.
\item Let $(S_0,R_0)$ be the finite $\Mod_g^b$-equivariant presentation for $\Torelli_g^b$ given
by Theorem \ref{theorem:torellipres}.
The $\Mod_g^b$-stabilizers of elements of $S_0$ are same as the $\Mod_g^b$-centralizers of these
elements, and Rafi--Selinger--Yampolsky \cite{Rafi} proved that the $\Mod_g^b$-centralizers of all elements of $\Mod_g^b$ are finitely
generated.  We remark that since elements of $\Torelli_g^b$ are ``pure'', this could be deduced for
elements of $S_0$ using much earlier results of Ivanov--McCarthy 
\cite{IvanovMcCarthy}.\footnote{Indeed, it is not hard to work these centralizers
out explicitly and verify directly that they are finitely generated.  The only
moderately difficult case is that of simply intersecting pair maps.}
\end{compactitem}
Theorem \ref{theorem:mainabstract} then implies that $\HH_2(\Torelli_g^b)$ is a finitely generated
$\Z[\Mod_g^b]$-module.  Since the action of $\Mod_g^b$ on $\HH_2(\Torelli_g^b)$ factors through
$\Sp_{2g}(\Z)$, this implies that $\HH_2(\Torelli_g^b)$ is a finitely generated $\Z[\Sp_{2g}(\Z)]$-module,
as desired.

\p{Generators}
A presentation for $\Torelli_g^b$ close to the one claimed by Theorem \ref{theorem:torellipres} was
constructed by Putman \cite{PutmanInfinite}.  To explain what remains to be done, we discuss
Putman's presentation.  We begin with the following (see Figure \ref{figure:generators}):
\vspace{-\baselineskip}
\begin{compactitem}
\item A {\em separating twist} is a Dehn twist $T_x$ with $x$ a separating simple closed curve.
\item A {\em bounding pair map} is a product $T_y T_z^{-1}$, with $y$ and $z$ disjoint simple closed
curves whose union separates $\Sigma_g^b$ (note that we do not require that $y$ and $z$ be nonseparating).
\item A {\em simply intersecting pair map} is a commutator $[T_a,T_b]$, where $a$ and $b$ are simple
closed curves that intersect twice with opposite signs (again, note that we do 
not require that $a$ or $b$ be nonseparating).
\end{compactitem}
\vspace{-\baselineskip}
These all lie in $\Torelli_g^b$.  Building on work of Birman \cite{BirmanSiegel}, Powell \cite{Powell}
proved that $\Torelli_g^b$ is generated by separating twists and bounding pair maps.  See
\cite{PutmanCutPaste, HatcherMargalit} for modern proofs.

Putman's generating set $S(\Torelli)$ consists
of all separating twists, all bounding pair maps, and all simply intersecting pair maps.  All of these
can be embedded into the surface in various ways, but the change of coordinates
principle from \cite[\S 1.3]{FarbMargalitPrimer} shows that up to the action of $\Mod_g^b$, there are
only finitely many ways to embed each into the surface.  In other words, the conjugation action of $\Mod_g^b$
on $\Torelli_g^b$ restricts to a $\Mod_g^b$-action on $S(\Torelli)$ with finitely many
orbits.  Let $S_0(\Torelli) \subset S(\Torelli)$ contain a single representative of each of these orbits,
so $S_0(\Torelli)$ is a finite set whose $\Mod_g^b$-orbit is the generating set $S(\Torelli)$.

\Figure{figure:generators}{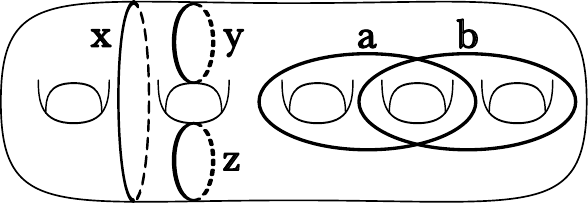}{A separating twist $T_x$, a bounding map $T_y T_z^{-1}$, and a simply
intersecting pair map $[T_a,T_b]$.}{60}

\p{Relations}
Let $R(\Torelli)$ be Putman's set of relations.
The element of $R(\Torelli)$ fall into a number of families: there are 8 ``formal relations'' along
with the ``lantern relations'', the ``crossed lantern relations'', the ``Witt--Hall relations'', and the
``commutator shuffle relations''.  For example, the lantern and crossed lantern relations are depicted in
Figure \ref{figure:lanterns}.  Almost all of these relations correspond to a finite list of pictures
that can be embedded in the surface in various ways, and again the change of coordinates principle
says that up to the action of $\Mod_g^b$ there are only finitely many ways of embedding each into the surface.

This might lead one to think that the action of $\Mod_g^b$ on $R(\Torelli)$ has finitely many orbits.  If
this were the case, then letting $R_0(\Torelli)$ be a set containing a single representative of each of
these orbits, the pair $(S_0(\Torelli),R_0(\Torelli))$ would be a finite
$\Mod_g^b$-equivariant presentation for $\Torelli_g^b$.

\p{Trouble}
However, there is an issue: three of the formal relations are not of this form.  In the notation of
\cite{PutmanInfinite}, these are the relations (F.6), (F.7), and (F.8).  They can be stated as follows.
Consider $M \in S(\Torelli)$.
\vspace{-\baselineskip}
\begin{compactitem}
\item Relation (F.6) says that if $T_x$ is a separating twist, then 
\[M T_x M^{-1} = T_{M(x)}.\]
\item Relation (F.7) says that if $T_y T_z^{-1}$ is a bounding pair map, then 
\[M T_x T_y^{-1} M^{-1} = T_{M(x)} T_{M(y)}^{-1}.\]
\item Relation (F.8) says that if $[T_a,T_b]$ is a simply intersecting pair map, then
\[M [T_a,T_b] M^{-1} = [T_{M(a)}, T_{M(b)}].\]
\end{compactitem}
\vspace{-\baselineskip}
Since there is no bound on how the curves making up $M$ intersect $T_x$ or $T_y T_z^{-1}$
or $[T_a,T_b]$, even up to the action of $\Mod_g^b$ these relations fall into 
infinitely many families.

\Figure{figure:lanterns}{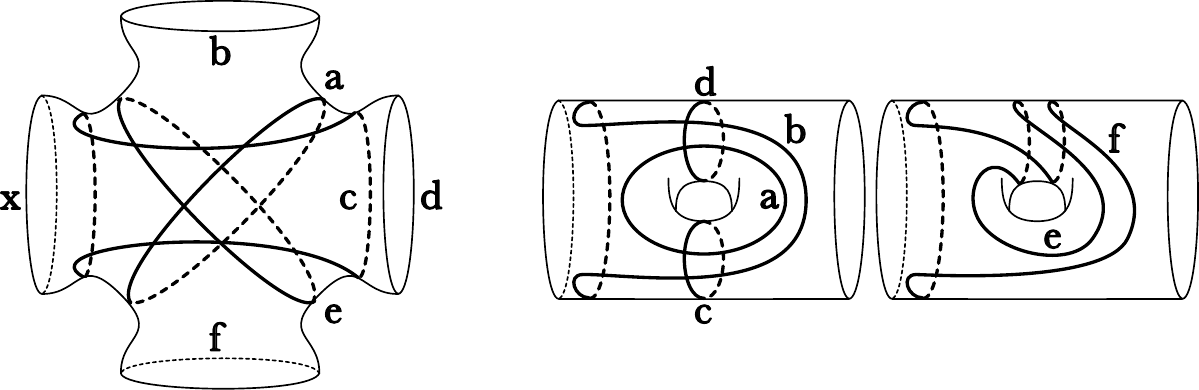}{On the left is the lantern relation $T_x = (T_a T_b^{-1}) (T_c T_d^{-1}) (T_e T_f^{-1})$,
where $T_x$ is a separating twist and the other three terms are bounding pair maps.
On the right is the crossed lantern relation $(T_a T_b^{-1}) (T_c T_d^{-1}) = T_e T_f^{-1}$; all three terms are
bounding pair maps.}{60}

\p{Weakly-finite equivariant presentations}
In summary, what Putman constructed in \cite{PutmanInfinite} is
the following kind of equivariant presentation (a priori weaker
than a finite one): 

\begin{definition}
Let $\Gamma$ be a group and let $G \lhd \Gamma$, so $\Gamma$ acts on $G$ by
conjugation.  For $\gamma \in \Gamma$ and $g \in G$, write
$\leftidx{^\gamma}{g} = \gamma g \gamma^{-1}$. 
A {\em weakly-finite $\Gamma$-equivariant
presentation} for $G$ consists of a pair $(S_0,R_0)$ as follows:
\vspace{-\baselineskip}
\begin{compactitem}
\item $S_0 \subset G$ is a finite set whose orbits 
\[S \coloneqq \Gamma \cdot S_0 = \Set{$\leftidx{^\gamma}{s}$}{$\gamma \in \Gamma$, $s \in S$} \subset G\]
generate $G$.  
\item $R_0 \subset F(S)$ is a finite set of relations for $G$ with the following property.  
Let $R \subset F(S)$ be $\Gamma \cdot R_0$, and define
\[R_{\conj} = \Set{$s t s^{-1} \left(\leftidx{^s}{t}\right)^{-1}$}{$s, t \in S$}.\]
See Remark \ref{remark:conj} below for more explanation of the form of these relations.
Then $R \cup R_{\conj}$ is a complete set of relations for $G$.\qedhere
\end{compactitem}
\end{definition}

\begin{remark}
\label{remark:conj}
The relations in $R_{\conj}$ might seem a bit puzzling.  To unpack them, observe first that
by its very definition $S \subset G$ is closed under conjugation by $\Gamma$.  Since $S \subset G \subset \Gamma$,
this implies that $S$ is closed under conjugation by elements of $S$, i.e.\ for $s,t \in S$ we have
$\leftidx{^s}{t} \in S$.  The elements $s$ and $t$ and $\leftidx{^s}{t}$ are all elements of $S \subset G$ that
a priori have no relationship in the free group $F(S)$.  The relations in $R_{\conj}$ (which hold
trivially in $G$) force $s \in F(S)$ to conjugate $t \in F(S)$ to $\leftidx{^s}{t} \in F(S)$.
\end{remark}

To fix this problem, we will show that in favorable situations a weakly-finite $\Gamma$-equivariant presentation
can be converted into a finite $\Gamma$-equivariant presentation:

\begin{maintheorem}
\label{theorem:deweak}
Let $\Gamma$ be a group and let $G \lhd \Gamma$.  Assume the following holds:
\vspace{-\baselineskip}
\begin{compactitem}
\item There exists a weakly-finite $\Gamma$-equivariant presentation for $G$.
\item Both $\Gamma$ and $G$ are finitely generated.
\end{compactitem}
\vspace{-\baselineskip}
Then there exists a finite $\Gamma$-equivariant presentation for $G$.
\end{maintheorem}

\vspace{-\baselineskip}
Since $\Mod_g^b$ is finitely generated for all $g$ and $b$ (a theorem due essentially to Dehn;
see \cite{FarbMargalitPrimer}) and $\Torelli_g^b$ is finitely generated for $g \geq 3$ and $b \in \{0,1\}$
(a theorem of Johnson \cite{JohnsonFinite}), we can combine Theorem \ref{theorem:deweak} with Putman's
result from \cite{PutmanInfinite} to deduce Theorem \ref{theorem:torellipres}, which asserts
that $\Torelli_g^b$ has a finite $\Mod_g^b$-equivariant presentation for $g \geq 3$ and $b \in \{0,1\}$.  
As we discussed above, this implies Theorem \ref{theorem:mainfg}.

\begin{remark}
Putman's construction of a weakly-finite $\Mod_g^b$-equivariant presentation for
$\Torelli_g^b$ used the main result of \cite{PutmanPresentation}, which explains
how to find presentations for groups acting on simplicial complexes without identifying
a fundamental domain for the action.  This machine inherently gives weakly-finite
equivariant presentations; indeed, the relations $R_{\conj}$ in the above definition
are precisely the ``conjugation relations'' from \cite{PutmanPresentation}.  Because
of this, we expect that Theorem \ref{theorem:deweak} will prove useful in other
applications of combinatorial group theory to representation stability.
\end{remark}

\p{Outline}
The two results we must prove are Theorems \ref{theorem:mainabstract} and \ref{theorem:deweak}.
The proof of Theorem \ref{theorem:mainabstract} is in \S \ref{section:abstract} and the proof of Theorem \ref{theorem:deweak} is in \S \ref{section:deweak}.  Here are brief
descriptions of how those proofs go:
\vspace{-\baselineskip}
\begin{compactitem}
\item For Theorem \ref{theorem:mainabstract}, the main idea is to embed
$\HH_2(G)$ into the five-term exact sequence in group homology and study
the finiteness properties of the various terms of this sequence as
$\Z[\Gamma]$-modules.
\item For Theorem \ref{theorem:deweak}, the main idea is to show that
all the conjugation relations can be deduced from those that show how to
express the conjugate of a $G$-generator by a $\Gamma$-generator
in terms of the finite generating set for $G$.
\end{compactitem}

\p{Acknowledgments}
We would like to thank Benjamin Steinberg for his help with the proof of Lemma \ref{lemma:finitedimrep}.

\section{The second homology group of normal subgroups}
\label{section:abstract}

This section contains the proof of Theorem \ref{theorem:mainabstract}.  We begin with some preliminary
results in \S \ref{section:abstractprelim} and then give the proof in \S \ref{section:abstractproof}.

\subsection{Preliminaries on finiteness conditions}
\label{section:abstractprelim}

All our rings have a unit, and unless otherwise specified all modules are left modules.  
We start with the following definition.

\begin{definition}
Let $R$ be a ring and let $M$ be an $R$-module.  We say that $M$ is of type $\FP_n$ if
there exists a length $n$ partial resolution
\[P_n \longrightarrow P_{n-1} \longrightarrow \cdots \longrightarrow P_0 \longrightarrow M \longrightarrow 0\]
of $M$ by finitely generated projective $R$-modules.  
\end{definition}

\begin{remark}
Being of type $\FP_0$ is equivalent to being finitely generated and being of type
$\FP_1$ is equivalent to being finitely presentable.
\end{remark}

The following lemma implies among other things that if $M$ has type $\FP_n$ for all $n$, then $M$
has an infinite length resolution by finitely generated projective $R$-modules.

\begin{lemma}[\cite{BieriBook}]
\label{lemma:alwaysfn}
Let $R$ be a ring and let $M$ be an $R$-module.  Then $M$ is of type $\FP_n$ if and only if $M$ is 
finitely generated and for all partial resolutions
\[P_{n'} \longrightarrow P_{n'-1} \longrightarrow \cdots \longrightarrow P_0 \longrightarrow M \longrightarrow 0\]
of $M$ by finitely generated projective $R$-modules of length $n' < n$, 
the kernel of the map $P_{n'} \rightarrow P_{n'-1}$ is finitely generated.
\end{lemma}

The following lemma describes how our finiteness conditions behave under extensions.

\begin{lemma}[\cite{BieriBook}]
\label{lemma:extension}
Let $R$ be a ring and let
\[0 \longrightarrow M' \longrightarrow M \longrightarrow M'' \longrightarrow 0\]
be a short exact sequence of $R$-modules.  The following hold.
\vspace{-\baselineskip}
\begin{compactenum}
\item If $M$ is of type $\FP_n$ and $M''$ is of type $\FP_m$ for some $n \geq 0$ and $m \geq 1$, 
then $M'$ is of type $\FP_r$ for $r = \min(n,m-1)$.
\item If $M'$ is of type $\FP_n$ and $M''$ is of type $\FP_m$ for some $n,m \geq 0$, then
$M$ is of type $\FP_r$ for $r = \min(n,m)$.
\item If $M'$ is of type $\FP_n$ and $M$ is of type $\FP_m$ for some $n,m \geq 0$,
then $M''$ is of type $\FP_r$ for $r = \min(n+1,m)$.
\end{compactenum}
\end{lemma}

We now turn to groups.

\begin{definition}
A group $G$ is of type $\FP_n$ if the trivial $\Z[G]$-module $\Z$ is of type $\FP_n$.
\end{definition}

For a group $G$ of type $\FP_n$, the following lemma give a large supply of $\Z[G]$-modules of type $\FP_{n}$.  We expect that this lemma is known to the experts, but we do not
know a reference.

\begin{lemma}
\label{lemma:finitedimrep}
Let $G$ be a group of type $\FP_{n}$ and let $M$ be a $\Z[G]$-module that is finitely generated
as an abelian group.  Then $M$ is of type $\FP_{n}$.
\end{lemma}
\vspace{-\baselineskip}
\begin{proof}
Let $M_{\tor}$ be the torsion subgroup of $M$, so $M/M_{\tor}$ is free abelian and
\[0 \longrightarrow M_{\tor} \longrightarrow M \longrightarrow M/M_{\tor} \longrightarrow 0\]
is a short exact sequence of $\Z[G]$-modules.  Lemma \ref{lemma:extension} says that it is
enough to prove that $M_{\tor}$ and $M/M_{\tor}$ are of type $\FP_{n}$, so we are reduced
to proving the following two special cases of the lemma.

\begin{case}
$M$ is a finite abelian group.
\end{case}

\vspace{-\baselineskip}
In this case, the kernel $G'$ of the map $G \rightarrow \Aut(M)$ is a finite-index subgroup of $G$.  Since
a partial resolution of $\Z$ by finitely generated projective $\Z[G]$-modules restricts to a
partial resolution of $\Z$ by finitely generated projective $\Z[G']$-modules, the group $G'$ is of
type $\FP_{n}$.  We claim that $M$ is of type $\FP_{n}$ as a $\Z[G']$-module.  Indeed, there exists
a finitely generated free abelian group $\tM$ and a surjection $f\colon \tM \rightarrow M$.
Endow $\tM$ with the trivial $\Z[G']$-module structure.  We then have a short exact sequence
\[0 \longrightarrow \ker(f) \longrightarrow \tM \stackrel{f}{\longrightarrow} M \longrightarrow 0\]
of $\Z[G']$-modules (all trivial!).  Both $\tM$ and $\ker(f)$ are direct sums of the trivial $\Z[G']$-module $\Z$.
Since $G'$ is of type $\FP_{n}$, so are the $\Z[G']$-modules $\tM$ and $\ker(f)$.  Lemma \ref{lemma:extension}
thus implies that $M$ is also of type $\FP_{n}$ as a $\Z[G']$-module, as claimed.

Now consider some $n'<n$ and a length $n'$ partial resolution
\begin{equation}
\label{eqn:partialunrestricted}
P_{n'} \longrightarrow P_{n'-1} \longrightarrow \cdots \longrightarrow P_0 \longrightarrow M \longrightarrow 0
\end{equation}
of $M$ by finitely generated projective $\Z[G]$-modules.  By Lemma \ref{lemma:alwaysfn}, to prove that $M$ is of
type $\FP_{n}$ as a $\Z[G]$-module, it is enough to prove that the kernel of the map $P_{n'} \rightarrow P_{n'-1}$
is a finitely generated $\Z[G]$-module.  Since $G'$ is a finite-index subgroup of $G$, the restriction
of \eqref{eqn:partialunrestricted} to $\Z[G']$ is a length $n'$ partial resolution of $M$ by finitely generated
projective $\Z[G']$-modules.  Since $M$ is of type $\FP_{n}$ as a $\Z[G']$-module, Lemma \ref{lemma:alwaysfn}
implies that the kernel of the map $P_{n'} \rightarrow P_{n'-1}$ is a finitely generated $\Z[G']$-module.  This
clearly implies that it is also a finitely generated $\Z[G]$-module, as desired. 

\begin{case}
$M$ is a finitely generated free abelian group.
\end{case}

\vspace{-\baselineskip}
We learned the argument in this case from Steinberg \cite{SteinbergAnswer}.  Recall that if $R$ is a ring, then a 
free $R$-bimodule is a direct sum of copies of the $R$-bimodule $R \otimes_{\Z} R$.  Since
$G$ is of type $\FP_{n}$, a theorem of 
Pride \cite[Theorem 2]{Pride} shows that there exists a length $n$ partial resolution
\[Q_n \longrightarrow Q_{n-1} \longrightarrow \cdots \longrightarrow Q_0 \longrightarrow \Z[G] \longrightarrow 0\]
of $\Z[G]$ by finitely generated free $\Z[G]$-bimodules.  We claim that
\begin{equation}
\label{eqn:proposedresolution}
Q_n \otimes_{\Z[G]} M \longrightarrow \cdots \longrightarrow Q_0 \otimes_{\Z[G]} M \longrightarrow \Z[G] \otimes_{\Z[G]} M = M \longrightarrow 0
\end{equation}
is a length $n$ partial resolution of $M$ by finitely generated free $\Z[G]$-modules.  This requires checking two things:
\vspace{-\baselineskip}
\begin{compactitem}
\item The chain complex \eqref{eqn:proposedresolution} is exact.  To see this, observe that each free
$\Z[G]$-bimodule $Q_k$ is a free right $\Z[G]$-module (not necessarily finitely generated).  This means
that the chain complex \eqref{eqn:proposedresolution} computes $\Tor^{\Z[G]}(\Z[G],M)$, which vanishes
since $\Z[G]$ is free.
\item Each $Q_k \otimes_{\Z[G]} M$ is a finitely generated free $\Z[G]$-module.  This follows from the fact
that $Q_k$ is a finite direct sum of copies of $\Z[G] \otimes_{\Z} \Z[G]$ together with the observation that
\[\left(\Z[G] \otimes_{\Z} \Z[G]\right) \otimes_{\Z[G]} M \cong \Z[G] \otimes_{\Z} M \cong \left(\Z[G]\right)^{\oplus \rank(M)}.\qedhere\] 
\end{compactitem}
\vspace{-\baselineskip}
\end{proof}

\subsection{From equivariant presentations to finiteness}
\label{section:abstractproof}

We now prove Theorem \ref{theorem:mainabstract}.

\vspace{-\baselineskip}
\begin{proof}[Proof of Theorem \ref{theorem:mainabstract}]
We first recall the setup.  Let $G$ and $\Gamma$ be groups such that $\Gamma$ acts on $G$.
Assume that the following hold:
\vspace{-\baselineskip}
\begin{compactitem}
\item $G$ has a finite $\Gamma$-equivariant presentation $(S_0,R_0)$.
\item $\HH_1(G)$ is finitely generated as an abelian group.
\item $\Gamma$ is of type $\FP_2$.
\item The $\Gamma$-stabilizers of all elements of $S_0$ are finitely generated.
\end{compactitem}
\vspace{-\baselineskip}
We must prove that $\HH_2(G)$ is finitely generated as a $\Z[\Gamma]$-module, i.e.\ that $\HH_2(G)$ is a 
$\Z[\Gamma]$-module of type $\FP_0$.

Let $S = \Gamma \cdot S_0$, let $R = \Gamma \cdot R_0$,
and let $\Norm{R}$ be the normal closure of $R$ in $F(S)$.  We
thus have a short exact sequence
\[1 \longrightarrow \Norm{R} \longrightarrow F(S) \longrightarrow G \longrightarrow 1.\]
Since $H_2(F(S)) = 0$, the five-term exact sequence in group cohomology associated to this short exact sequence
is of the form
\[0 \longrightarrow \HH_2\left(G\right) \longrightarrow \left(\HH_1\left(\Norm{R}\right)\right)_G \stackrel{\phi}{\longrightarrow} \HH_1\left(F\left(S\right)\right) \stackrel{\psi}{\longrightarrow} \HH_1\left(G\right) \longrightarrow 0.\]
By construction, this is an exact sequence of $\Z[\Gamma]$-modules.  The following five claims
elucidate the finiteness properties of various terms of this exact sequence.  The theorem itself
is the fifth one.

\begin{claimsa}
\label{claimsa.1}
$\HH_1(G)$ is a $\Z[\Gamma]$-module of type $\FP_{2}$.
\end{claimsa}

\vspace{-\baselineskip}
This follows from Lemma \ref{lemma:finitedimrep}, which
we can apply since $\Gamma$ is a group of type $\FP_{2}$ and $\HH_1(G)$ is a finitely generated abelian group.

\begin{claimsa}
\label{claimsa.2}
$\HH_1(F(S))$ is a $\Z[\Gamma]$-module of type $\FP_{1}$.
\end{claimsa}

\vspace{-\baselineskip}
By construction, we have
\[\HH_1(F(S)) \cong \bigoplus_{s \in S_0} \Z[\Gamma / \Gamma_s],\]
where $\Gamma_s$ denotes the $\Gamma$-stabilizer of $s \in S_0$.  Fixing some $s \in S_0$, it
is thus enough to prove that $\Z[\Gamma / \Gamma_s]$ is a $\Z[\Gamma]$-module of type $\FP_1$.  By assumption,
$\Gamma_s$ is finitely generated; let $X$ be a finite generating set for it.  We then have a finite presentation
\[\bigoplus_{x \in X} \Z[\Gamma] \stackrel{\iota}{\longrightarrow} \Z[\Gamma] \longrightarrow \Z[\Gamma/\Gamma_s] \longrightarrow 0,\]
where the map $\iota$ is of the form $\iota = \oplus_{x \in X} \iota_x$ with 
$\iota_x\colon \Z[\Gamma] \rightarrow \Z[\Gamma]$ the map taking $\omega \in \Z[\Gamma]$ to $\omega (x-1) \in \Z[\Gamma]$.
The claim follows.

\begin{claimsa}
\label{claimsa.3}
$\ker(\psi) = \Image(\phi)$ is a $\Z[\Gamma]$-module of type $\FP_1$.
\end{claimsa}

\vspace{-\baselineskip}
This follows from Claims \ref{claimsa.1} and \ref{claimsa.2}
together with Lemma \ref{lemma:extension}.

\begin{claimsa}
\label{claimsa.4}
$\left(\HH_1\left(\Norm{R}\right)\right)_G$ is a $\Z[\Gamma]$-module of type $\FP_0$.
\end{claimsa}

\vspace{-\baselineskip}
Since $R_0$ is finite, it is enough to show that the evident map
\[\bigoplus_{r \in R_0} \Z[\Gamma] \longrightarrow \left(\HH_1\left(\Norm{R}\right)\right)_G\]
is a surjection.  The image of this map equals the image of the map
\[\HH_1(R) \longrightarrow \left(\HH_1\left(\Norm{R}\right)\right)_G.\]
But this map is surjective since we are taking $G$-coinvariants, which causes all $G$-conjugates
of an element $r \in R$ to collapse to a single element of $\left(\HH_1\left(\Norm{R}\right)\right)_G$.

\begin{claimsa}
\label{claimsa.5}
$\ker(\phi) = \HH_2(G)$ is a $\Z[\Gamma]$-module of type $\FP_0$.
\end{claimsa}

\vspace{-\baselineskip}
This follows from Claims \ref{claimsa.3} and \ref{claimsa.4} together with Lemma \ref{lemma:extension}.
\end{proof}

\section{Upgrading weakly-finite equivariant presentations}
\label{section:deweak}

We conclude the paper by proving Theorem \ref{theorem:deweak}.

\vspace{-\baselineskip}
\begin{proof}[Proof of Theorem \ref{theorem:deweak}]
We begin by recalling the setup.
Let $\Gamma$ be a group and let $G$ be a normal subgroup of $\Gamma$.  Assume
that both $G$ and $\Gamma$ are finitely generated, and let $(S_0,R_0)$ be
a weakly-finite $\Gamma$-equivariant presentation for $G$.  Our goal is to
construct a finite $\Gamma$-equivariant presentation for $G$.

For $\gamma \in \Gamma$ and $g \in G$, write
$\leftidx{^\gamma}{g}$ for the image of $g$ under the action of $\gamma$.
As in the definition of a weakly-finite $\Gamma$-equivariant presentation, let
$S = \Gamma \cdot S_0$, let
$R = \Gamma \cdot R_0$, and let
\[R_{\conj} = \Set{$s t s^{-1} \left(\leftidx{^s}{t}\right)^{-1}$}{$s, t \in S$}.\]
By definition, $R \cup R_{\conj}$ is a complete set of relations for $G$.  To prove
the theorem, it is enough to construct a further finite set $R_0' \subset F(S)$
with the following property:
\begin{itemize}
\item[($\dagger$)] Let $R' = \Gamma \cdot R_0'$.
Then each relation in $R_{\conj}$ is a consequence of the relations in $R'$.
\end{itemize}
Since $G$ is finitely generated and $S$ generates $G$, there exists a finite subset
$X$ of $S$ that generates $G$.  Define $R_0''$ to be the following finite subset
of $R_{\conj}$:
\[R_0'' = \Set{$s x s^{-1} \left(\leftidx{^s}{x}\right)^{-1}$}{$s \in S_0$, $x \in X$}.\]
Now let $Y$ be a finite generating set for $\Gamma$ that is symmetric in the
sense that if $y \in Y$, then $y^{-1} \in Y$.  For $y \in Y$ and $x \in X$,
we can find some $w_{y,x} \in F(X) \subset F(S)$ that maps to $\leftidx{^y}{x} \in S \subset G$ under
the composition $F(X) \hookrightarrow F(S) \rightarrow G$.
Define 
\[R_0''' = \Set{$\left(\leftidx{^y}{x}\right) \left(w_{y,x}\right)^{-1}$}{$y \in Y$, $x \in X$} \subset F(S)\]
and $R_0' = R_0'' \cup R_0'''$.  

We claim that $R_0'$ satisfies ($\dagger$).  To see this, let $R' = \Gamma \cdot R_0'$.
The claim ($\dagger$) is
the third of the following claims.  For words $h,k \in F(S)$, write $h \equiv k$ if
$h$ equals $k$ modulo $R'$.

\begin{claimsb}
\label{claimsb:1}
For all $u \in S$, there is some word $w \in F(X)$ such that $u \equiv w$.
\end{claimsb}

\vspace{-\baselineskip}
There exists some $u_0 \in S_0$ and $\gamma \in \Gamma$ such that 
\[u = \leftidx{^{\gamma}}{u_0}.\]
The proof of the claim will be 
by induction on the length of the shortest word in the generating set $Y$ for $\Gamma$ needed to write $\gamma$.  The
base case where that word has length $0$ is trivial, so assume that it has positive length.  Using the
fact that $Y$ is symmetric, we can write $\gamma = y \gamma'$, where $y \in Y$ and $\gamma'$ can be 
written as a shorter word than $\gamma$.  Our inductive hypothesis say that there exists some
$w' \in F(X)$ such that
\[\leftidx{^{\gamma'}}{u_0} \equiv w'.\]
Since the relations in $R'$ are closed under the action of $\Gamma$, this implies that
\[u = \leftidx{^{y \gamma'}}{u_0} \equiv \leftidx{^y}{w'}.\]
Now write $w = x_1^{e_1} \cdots x_n^{e_n}$ with $x_i \in X$ and $e_i \in \{\pm 1\}$.  We then have
\[\leftidx{^y}{w'} = \left(\leftidx{^y}{x_1}\right)^{e_1} \left(\leftidx{^y}{x_2}\right)^{e_2} \cdots \left(\leftidx{^y}{x_n}\right)^{e_n}.\]
Using the relations in $R_0''' \subset R'$, we see that each term on the right hand side of the previous equation
is equivalent modulo the relations in $R'$ to a word in $F(X)$.  The claim follows.

\begin{claimsb}
\label{claimsb:2}
For all $s_0 \in S_0$ and $u \in S$, we have
\begin{equation}
\label{eqn:claimsb2.1}
s_0 u s_0^{-1} \equiv \leftidx{^{s_0}}{u}.
\end{equation}
\end{claimsb}

\vspace{-\baselineskip}
Using Claim \ref{claimsb:1}, we can find a word $w \in F(X)$ such that 
\begin{equation}
\label{eqn:claimsb2.2}
u \equiv w.
\end{equation}
Since the relations in $R'$ are closed under the action of $\Gamma$, this implies that
\begin{equation}
\label{eqn:claimsb2.3}
\leftidx{^{s_0}}{u} \equiv \leftidx{^{s_0}}{w}.
\end{equation}
Write $w = x_1^{e_1} \cdots x_n^{e_n}$ with $x_i \in X$ and $e_i \in \{\pm 1\}$.  The equation
\eqref{eqn:claimsb2.1} can then be deduced as follows:
\begin{align*}
s_0 u s_0^{-1} &\equiv s_0 \left(x_1^{e_1} x_2^{e_2} \cdots x_n^{e_n}\right) s_0^{-1} \\
&= \left(s_0 x_1 s_0^{-1}\right)^{e_1} \left(s_0 x_2 s_0^{-1}\right)^{e_2} \cdots \left(s_0 x_n s_0^{-1}\right)^{e_n} \\
&\equiv \left(\leftidx{^{s_0}}{x_1}\right)^{e_1} \left(\leftidx{^{s_0}}{x_2}\right)^{e_2} \cdots \left(\leftidx{^{s_0}}{x_n}\right)^{e_n} \\
&= \leftidx{^{s_0}}{\left(x_1^{e_1} x_2^{e_2} \cdots x_n^{e_n}\right)} \\
&= \leftidx{^{s_0}}{w} \\
&\equiv \leftidx{^{s_0}}{u}.
\end{align*}
Here the first $\equiv$ is \eqref{eqn:claimsb2.2}, the second $\equiv$ uses the relations in $R_0'' \subset R$, and
the third $\equiv$ is \eqref{eqn:claimsb2.3}.

\begin{claimsb}
\label{claimsb:3}
For all $s,t \in S$, we have
\begin{equation}
\label{eqn:claimsb3.1}
s t s^{-1} \equiv \leftidx{^s}{t}.
\end{equation}
\end{claimsb}

\vspace{-\baselineskip}
We can write $s = \leftidx{^\gamma}{s_0}$ for some $\gamma \in \Gamma$
and $s_0 \in S_0$.  Define $u = \leftidx{^{\gamma^{-1}}}{t}$.  Claim \ref{claimsb:2} implies that
\begin{equation}
\label{eqn:claimsb3.2}
s_0 u s_0^{-1} \equiv \leftidx{^{s_0}}{u}
\end{equation}
The relations in $R'$ are closed under the action of $\Gamma$, so we can apply $\gamma$ to both sides
of \eqref{eqn:claimsb3.2} and deduce that
\begin{equation}
\label{eqn:claimsb3.3}
\leftidx{^{\gamma}}{\left(s_0 u s_0^{-1}\right)} \equiv \leftidx{^{\gamma s_0}}{u}.
\end{equation}
The left hand side of \eqref{eqn:claimsb3.3} is
\begin{equation}
\label{eqn:claimsb3.4}
\leftidx{^{\gamma}}{\left(s_0 u s_0^{-1}\right)} = \left(\leftidx{^\gamma}{s_0}\right) \left(\leftidx{^{\gamma \gamma^{-1}}}{t}\right) \left(\leftidx{^\gamma}{s_0}\right)^{-1} = s t s^{-1},
\end{equation}
while the right hand side of \eqref{eqn:claimsb3.3} is
\begin{equation}
\label{eqn:claimsb3.5}
\leftidx{^{\gamma s_0}}{u} = \leftidx{^{\gamma s_0 \gamma^{-1}}}{t} = \leftidx{^s}{t};
\end{equation}
here we are using the fact that $\Gamma$ acts on $G$ by conjugation, so
\[\gamma s_0 \gamma^{-1} = \leftidx{^\gamma}{s_0} = s\]
in $G$.  
Since \eqref{eqn:claimsb3.4} and \eqref{eqn:claimsb3.5} are also the left and right hand sides of \eqref{eqn:claimsb3.1}, we
conclude that \eqref{eqn:claimsb3.1} holds, as desired.

Claim \ref{claimsb:3} was precisely ($\dagger$) above, which as we noted implies that
$(S_0,R_0 \cup R_0')$ is a finite $\Gamma$-equivariant presentation of $G$.  The theorem follows.
\end{proof}

\begin{footnotesize}
\begin{tabular*}{\linewidth}[t]{@{}p{\widthof{Department of Mathematics}+0.5in}@{}p{\linewidth - \widthof{Department of Mathematics} - 0.5in}@{}}
{\raggedright
Martin Kassabov\\
Department of Mathematics\\ 
Malott Hall\\ 
Cornell University\\ 
Ithaca, New York 14850\\
{\tt kassabov@math.cornell.edu}}
&
{\raggedright
Andrew Putman\\
Department of Mathematics\\
University of Notre Dame \\
255 Hurley Hall\\
Notre Dame, IN 46556\\
{\tt andyp@nd.edu}}
\end{tabular*}\hfill
\end{footnotesize}

\end{document}